\documentclass[12pt,reqno,a4paper]{amsart}
\usepackage{amsfonts}
\usepackage{amsmath}

\newtheorem{theorem}{Theorem}[section]
\newtheorem{corollary}[theorem]{Corollary}
\newtheorem{lemma}[theorem]{Lemma}
\theoremstyle{definition}

\theoremstyle{remark} \theoremstyle{example}
\newtheorem{example}[theorem]{Example}

\numberwithin{equation}{section}
\input tcilatex

\begin{document}
\title[Starlikeness and convexity of integral operators$\cdots $]{%
Starlikeness and convexity of integral operators involving Mittag-Leffler
functions}
\author{B.A. Frasin}
\address{Department of Mathematics, Faculty of Science, Al al-Bayt
University, Mafraq, Jordan.}
\email{bafrasin@yahoo.com.}

\begin{abstract}
In this paper, we shall find the order of starlikeness and convexity for
integral operators%
\begin{equation*}
\mathbb{F}_{\alpha _{j},\beta _{j},\lambda _{j},\zeta }(z)=\left\{ \zeta
\int\limits_{0}^{z}t^{\zeta -1}\prod_{j=1}^{n}\left( \frac{\mathbb{E}%
_{\alpha _{j},\beta _{j}}(t)}{t}\right) ^{1/\lambda _{j}}dt\right\}
^{1/\zeta },
\end{equation*}%
where the functions $\mathbb{E}_{\alpha _{j},\beta _{j}}$ are the normalized
Mittag-Leffler functions.

\textbf{Keywords: }Analytic functions; Starlike and convex functions;
Integral operators; Mittag-Leffler function.

\textbf{2010 Mathematics Subject Classification}: 33E12, 30C45.
\end{abstract}

\maketitle

\section{\textbf{Introduction and preliminaries}}

Let $\mathcal{A}$ denote the class of functions of the form :

\begin{equation}
f(z)=z+\sum\limits_{n=2}^{\infty }a_{n}z^{n},
\end{equation}%
which are analytic in the open unit disk $\mathbb{U}=\{z:\left\vert
z\right\vert <1\}$. A function $f(z)$ $\in \mathcal{A}$ is said to be
starlike of order $\delta $ if it satisfies 
\begin{equation}
\mbox{Re}\left( \frac{zf^{\prime }(z)}{f(z)}\right) >\delta \qquad (z\in 
\mathbb{U})
\end{equation}%
for some $\delta (0\leq \delta <1).$ We denote by $\mathcal{S}^{\ast
}(\delta )$ the subclass of $\mathcal{A}$ consisting of functions which are
starlike of order $\delta $ in $\mathbb{U}.$ Clearly $\mathcal{S}^{\ast
}(\delta )\subseteq \mathcal{S}^{\ast }(0)=\mathcal{S}^{\ast },$ where $%
\mathcal{S}^{\ast }$is the class of functions that are starlike in $\mathbb{U%
}$ $.$ Also, a function $f(z)$ $\in \mathcal{A}$ is said to be convex of
order $\delta $ if it satisfies

\begin{equation}
\mbox{Re}\left( 1+\frac{zf^{\prime \prime }(z)}{f^{\prime }(z)}\right)
>\delta \qquad (z\in \mathbb{U})
\end{equation}%
for some $\delta (0\leq \delta <1).$ We denote by $\mathcal{C(\delta )}$ the
subclass of $\mathcal{A}$ consisting of functions which are convex of order $%
\alpha $ in $\mathbb{U}$. Clearly $\mathcal{C(\delta )\subseteq C}(0)=%
\mathcal{C},$ the class of functions that are convex in $\mathbb{U}$.

\bigskip Let $E_{\alpha }(z)$ be the function defined by%
\begin{equation*}
E_{\alpha }(z)=\sum_{n=0}^{\infty }\frac{z^{n}}{\Gamma (\alpha n+1)},\quad
(z\in \mathbb{C},\mathfrak{R(\alpha )}>0\;).
\end{equation*}

The function $E_{\alpha }(z)$ was introduced by Mittag-Leffler \cite{mit}
and is, therefore, known as the Mittag-Leffler function. \bigskip A more
general function $E_{\alpha ,\beta }$ generalizing $E_{\alpha }(z)$ was
introduced by Wiman \cite{wi1,wi2} and defined by

\begin{equation}
E_{\alpha ,\beta }(z)=\sum_{n=0}^{\infty }\frac{z^{n}}{\Gamma (\alpha
n+\beta )},\quad (z,\alpha ,\beta \in \mathbb{C},\mathfrak{R(\alpha )}>0\;).
\end{equation}

The Mittag-Leffler function arises naturally in the solution of fractional
order differential and integral equations, and especially in the
investigations of fractional generalization of kinetic equation, random
walks, L\'{e}vy flights, super-diffusive transport and in the study of
complex systems. Several properties of Mittag-Leffler function and
generalized Mittag-Leffler function can be found e.g. in (\cite{ade, gar}, 
\cite{kir1}-\cite{sri}).

\bigskip Observe that Mittag-Leffler function $E_{\alpha ,\beta }$ does not
belong to the family $\mathcal{A}$. Therefore, we consider the following
normalization of the Mittag-Leffler function:

\begin{eqnarray}
\mathbb{E}_{\alpha ,\beta }(z) &=&\Gamma (\beta )zE_{\alpha ,\beta }(z) 
\notag \\
&=&z+\sum_{n=2}^{\infty }\frac{\Gamma (\beta )}{\Gamma (\alpha (n-1)+\beta )}%
z^{n},  \label{44}
\end{eqnarray}

where $z,\alpha ,\beta \in \mathbb{C};\beta \neq 0,-1,-2,\cdots $ and $%
\mathfrak{R(\alpha )}>0.$

\bigskip Whilst formula \eqref{44} holds for complex-valued $\alpha ,\beta $
and $z\in \mathbb{C}$, however in this paper, we shall restrict our
attention to the case of real-valued $\alpha ,\beta $ and $z\in $ $\mathbb{U}
$. Observe that the function $\mathbb{E}_{\alpha ,\beta }$ contains many
well-known functions as its special case, for example, $\mathbb{E}%
_{2,1}(z)=z\cosh \sqrt{z}$, $\mathbb{E}_{2,2}(z)=\sqrt{z}\sinh \sqrt{z}$, $%
\mathbb{E}_{2,3}(z)=2[\cosh \sqrt{z}-1]\ $and $\mathbb{E}_{2,4}(z)=6[\sinh 
\sqrt{z}-\sqrt{z}]/\sqrt{z}.$

Geometric properties including starlikeness, convexity and
close-to-convexity for the Mittag-Leffler function $E_{\alpha ,\beta }$ were
recently investigated by Bansal and Prajapat in \cite{ban}.

Very recently, Srivastava \textit{et al.}\cite{sribv} introduced a new
integral operator $\mathbb{F}_{\alpha _{j},\beta _{j},\lambda _{j},\zeta }$
involving Mittag-Leffler functions given by 
\begin{equation}
\mathbb{F(}z)=\mathbb{F}_{\alpha _{j},\beta _{j},\lambda _{j},\zeta
}(z)=\left\{ \zeta \int\limits_{0}^{z}t^{\zeta -1}\prod_{j=1}^{n}\left( 
\frac{\mathbb{E}_{\alpha _{j},\beta _{j}}(t)}{t}\right) ^{1/\lambda
_{j}}dt\right\} ^{1/\zeta },  \label{z1}
\end{equation}

\bigskip where the functions $\mathbb{E}_{\alpha _{j},\beta _{j}}$ are the
normalized Mittag-Leffler functions defined by%
\begin{equation*}
\mathbb{E}_{\alpha _{j},\beta _{j}}(z)=\Gamma (\beta _{j})zE_{\alpha
_{j},\beta _{j}}(z).
\end{equation*}%
and the parameters $\lambda _{1},\lambda _{1},\dots ,\lambda _{n}$ and $%
\zeta $ are are positive real numbers such that the integrals in \eqref{z1}
exist. Here and throughout in the sequel every many-valued function is taken
with the principal branch.

In the present paper, we will find the order of starlikeness and convexity
for the above integral defined by \eqref{z1}.

\bigskip\ \ In order to prove our main results, we recall the following
lemmas.

\begin{lemma}
\label{lem1} ( \cite{milmoc}). Let $\Phi (u,v)$\ be a complex valued
function,%
\begin{equation*}
\Phi :\mathbb{D}\rightarrow \mathbb{C},\mathbb{\qquad (D\subset {C}}^{2})
\end{equation*}%
and let $u=u_{1}+iu_{2}$\ and $v=v_{1}+iv_{2}.$\ Suppose that the function $%
\Phi (u,v)$\ satisfies
\end{lemma}

$(i)$ $\Phi (u,v)$\ \textit{is continuous in} $\mathbb{D};$

$(ii)$ $(1,0)\in \mathbb{D}$\ \textit{and} $\mbox{Re}(\Phi (1,0))>0;$

$(iii)$ $\mbox{Re}(\Phi (iu_{2},v_{1}))\leq 0$\ \textit{for all} $%
(iu_{2},v_{1})\in \mathbb{D}$\ \textit{and such that} $v_{1}\leq
-(1+u_{2}^{2})/2.$

\textit{Let} $p(z)=1+p_{1}z+p_{2}z^{2}+\cdots $\ \textit{be analytic in} $%
\mathbb{U}$\ \textit{such that} $(p(z),zp^{\prime }(z))\in \mathbb{D}$\ 
\textit{for all} $z\in \mathbb{U}.$\ \textit{If} $\mbox{Re}(\Phi
(p(z),zp^{\prime }(z)))>0\,\ (z\in \mathbb{U})$, \textit{then} $\mbox{Re}%
(p(z))>0\,$\ $(z\in \mathbb{U}).$

\begin{lemma}
\label{lem4}(\cite{ban})Let $\alpha \geq 1$ and $0\leq \eta <1$. Suppose
also that 
\begin{equation*}
\Psi (\eta )=\frac{(3-\eta )+\sqrt{5\eta ^{2}-18\eta +17}}{2(1-\eta )}.
\end{equation*}
\end{lemma}

\textit{If} $\beta \geq \Psi (\eta )$, \textit{then} $\mathbb{E}_{\alpha
,\beta }$ \textit{is starlike function of order} $\eta $.

\begin{lemma}
\label{lem3}( \cite{sribv})Let $\alpha \geq 1$ and $\beta \geq 1.$ Then 
\begin{equation}
\left\vert \frac{z\mathbb{E}_{\alpha ,\beta }^{\prime }(z)}{\mathbb{E}%
_{\alpha ,\beta }(z)}-1\right\vert \leq \frac{2\beta +1}{\beta ^{2}-\beta -1}%
,\qquad (z\in \mathbb{U)}.  \label{r4}
\end{equation}
\end{lemma}

\section{Main Results}

Our first result provides the order of starlikeness for integral operator of
the type (\ref{z1}).

\begin{theorem}
\label{th1}Let $\alpha _{j}\geq 1,0\leq \eta _{j}<1$, and 
\begin{equation*}
\beta _{j}\geq \frac{(3-\eta _{j})+\sqrt{5\eta _{j}^{2}-18\eta _{j}+17}}{%
2(1-\eta _{j})},
\end{equation*}%
for all $j=1,2,3,\ldots ,n$. Suppose also that $\lambda _{1},\lambda
_{2},\dots ,\lambda _{n},\zeta $ \textit{are positive real numbers such that}%
\begin{equation*}
\sum_{j=1}^{n}\frac{1-\eta _{j}}{\lambda _{j}}\leq \zeta ,
\end{equation*}%
\textit{then} $\mathbb{F(}z)\mathbb{\in }$ $\mathcal{S}^{\ast }(\delta ),$%
\textit{where}
\end{theorem}

\begin{equation}
\delta =\frac{-\left( \dsum\limits_{j=1}^{n}\frac{2(1-\eta _{j})}{\lambda
_{j}}-2\zeta +1\right) +\sqrt{\left( \dsum\limits_{j=1}^{n}\frac{2(1-\eta
_{j})}{\lambda _{j}}-2\zeta +1\right) ^{2}+8\zeta }}{4\zeta },\text{ \ \ }%
0\leq \delta <1.  \label{pp}
\end{equation}

\begin{proof}
Define the function $p(z)$ by

\begin{equation}
\frac{z\mathbb{F}^{\prime }(z)}{\mathbb{F}(z)}:=\delta +(1-\delta )p(z),
\label{b1}
\end{equation}%
where $\delta $ as given in (\ref{pp}).

Then $p(z)=1+b_{1}z+b_{2}z+\cdots $ is analytic in $\mathbb{U}$. It follows
from (\ref{z1}) and (\ref{b1})\ that%
\begin{equation}
\frac{z^{\zeta }\tprod\limits_{j=1}^{n}\left( \frac{\mathbb{E}_{\alpha
_{j},\beta _{j}}(z)}{z}\right) ^{1/\lambda _{i}}}{\mathbb{F}^{\zeta }(z)}%
=\delta +(1-\delta )p(z).  \label{b2}
\end{equation}%
Differentiating (\ref{b2}) logarithmically, we obtain

\begin{equation}
\sum_{j=1}^{n}\frac{1}{\lambda _{j}}\left( \frac{z\mathbb{E}_{\alpha
_{j},\beta _{j}}^{\prime }(z)}{\mathbb{E}_{\alpha _{j},\beta _{j}}(z)}%
\right) =\zeta (1-\delta )p(z)+\frac{(1-\delta )zp^{\prime }(z)}{\delta
+(1-\delta )p(z)}+\sum_{j=1}^{n}\frac{1}{\lambda _{j}}-\zeta (1-\delta ).
\end{equation}%
From Lemma \ref{lem1}, $\mathbb{E}_{\alpha _{j},\beta _{j}}$ is starlike
function of order $\eta _{j}$ for all $j=1,2,3,\ldots ,n$, therefore we have%
\begin{eqnarray}
&&\sum_{j=1}^{n}\frac{1}{\lambda _{j}}\mbox{Re}\left( \frac{z\mathbb{E}%
_{\alpha _{j},\beta _{j}}^{\prime }(z)}{\mathbb{E}_{\alpha _{j},\beta
_{j}}(z)}\right)  \notag \\
&=&\mbox{Re}\left\{ \zeta (1-\delta )p(z)+\frac{(1-\delta )zp^{\prime }(z)}{%
\delta +(1-\delta )p(z)}+\sum_{j=1}^{n}\frac{1-\eta _{j}}{\lambda _{j}}%
-\zeta (1-\delta )\right\} >0.
\end{eqnarray}%
If we define the function $\Phi (u,v)$ by%
\begin{equation}
\Phi (u,v)=\zeta (1-\delta )u+\frac{(1-\delta )v}{\delta +(1-\delta )u}%
+\sum_{j=1}^{n}\frac{1-\eta _{j}}{\lambda _{j}}-\zeta (1-\delta )
\end{equation}%
with $u=u_{1}+iu_{2}$ and $v=v_{1}+iv_{2},$ then

$(i)$ $\Phi (u,v)$ is continuous in $\mathbb{D}=\mathbb{C}^{2};$

$(ii)$ $(1,0)\in \mathbb{D}$ and $\mbox{Re}(\Phi
(1,0))=\dsum\limits_{j=1}^{n}\frac{1-\eta _{j}}{\lambda _{j}}>0;$

$(iii)$ For all $(iu_{2},v_{1})\in \mathbb{D}$ and such that $v_{1}\leq
-(1+u_{2}^{2})/2,$%
\begin{eqnarray}
\mbox{Re}(\Phi (iu_{2},v_{1})) &=&\frac{\delta (1-\delta )v_{1}}{\delta
^{2}+(1-\delta )^{2}u_{2}^{2}}+\sum_{j=1}^{n}\frac{1-\eta _{j}}{\lambda _{j}}%
-\zeta (1-\delta )  \notag \\
&\leq &\frac{A+Bu_{2}^{2}}{C}  \label{jj}
\end{eqnarray}

where%
\begin{equation*}
A=\delta \left( 2\zeta \delta ^{2}+\left( \sum_{j=1}^{n}\frac{2(1-\eta _{j})%
}{\lambda _{j}}-2\zeta +1\right) \delta -1\right) ,
\end{equation*}

\begin{equation*}
B=(1-\delta )^{2}\left( \sum_{j=1}^{n}\frac{2(1-\eta _{j})}{\lambda _{j}}%
-2\zeta (1-\delta )\right) -\delta (1-\delta ),
\end{equation*}

and%
\begin{equation*}
C=2\delta ^{2}+2(1-\delta )^{2}u_{2}^{2}.
\end{equation*}

The right hand side of (\ref{jj}) is negative if $A\leq 0$ and $B\leq 0$.
From $A\leq 0$, we have the value of $\delta $ given by (\ref{pp}) and from $%
B\leq 0$, we have $0\leq \delta <1$. Therefore, the function $\Phi (u,v)$
satisfies the conditions in Lemma \ref{lem1}. Thus we have $\mbox{Re}%
(p(z))>0\,(z\in \mathbb{U}),$ that is $\mathbb{F(}z)\mathbb{\in }$ $\mathcal{%
S}^{\ast }(\delta ).$
\end{proof}

\bigskip Let $n=1,$ $\alpha _{1}=\alpha ,$ $\beta _{1}=\beta ,$ $\lambda
_{1}=\lambda $ and $\eta _{1}=0$ in Theorem \ref{th1}, we have the following
result.

\begin{corollary}
\label{cor12}Let $\alpha \geq 1$ and $\beta \geq \frac{3+\sqrt{17}}{2}.$ T%
\textit{hen} $\mathbb{F}_{\alpha ,\beta ,\lambda ,\zeta }(z)=\left\{ \zeta
\int\limits_{0}^{z}t^{\zeta -1}\left( \frac{\mathbb{E}_{\alpha ,\beta }(t)}{t%
}\right) ^{1/\lambda }dt\right\} ^{1/\zeta }\mathbb{\in }$ $\mathcal{S}%
^{\ast }(\delta )$ \textit{where }$\lambda $ and $\zeta $ \textit{are
positive real numbers such that }$\frac{1}{\lambda }\leq \zeta ,$ and
\end{corollary}

\begin{equation}
\delta =\frac{-\left( \frac{2}{\lambda }-2\zeta +1\right) +\sqrt{\left( 
\frac{2}{\lambda }-2\zeta +1\right) ^{2}+8\zeta }}{4\zeta },\text{ \ \ }%
0\leq \delta <1.
\end{equation}

Putting $\lambda =1$ and $\zeta =1$ in Corollary \ref{cor12}, we immediately
have

\begin{corollary}
\label{cor44}Let $\alpha \geq 1$ and $\beta \geq \frac{3+\sqrt{17}}{2}.$ T%
\textit{hen} $\mathbb{F}_{\alpha ,\beta ,1,1}(z)=\int\limits_{0}^{z}\left( 
\frac{\mathbb{E}_{\alpha ,\beta }(t)}{t}\right) dt\mathbb{\ }$is starlike of
order $1/2$ in $\mathbb{U}$.
\end{corollary}

\begin{example}
$~$Let $\mathbb{E}_{2,4}(z)=6[\sinh \sqrt{z}-\sqrt{z}]/\sqrt{z},$ \textit{%
then}$\mathbb{\ }\int\limits_{0}^{z}\frac{6[\sinh \sqrt{t}-\sqrt{t}]}{t^{3/2}%
}dt$ is starlike of order $1/2$ in $\mathbb{U}$.
\end{example}

Making use Lemma \ref{lem3}, we determine the order of convexity for
integral operator of the type (\ref{z1}).

\begin{theorem}
\label{th2.1} Let $\alpha _{1},\alpha _{2},\dots ,\alpha _{n}\geq 1$, $\beta
_{1},\beta _{2},\dots ,\beta _{n}\geq \frac{1}{2}(1+\sqrt{5})$ and consider
the normalized Mittag-Leffler functions $\mathbb{E}_{\alpha _{j},\beta _{j}}$
defined by 
\begin{equation}
\mathbb{E}_{\alpha _{j},\beta _{j}}(z)=\Gamma (\beta _{j})zE_{\alpha
_{j},\beta _{j}}(z).  \label{kkk}
\end{equation}%
Let $\beta =\min \{\beta _{1},\beta _{2},\dots ,\beta _{n}\}$ and $\lambda
_{1},\lambda _{2},\dots ,\lambda _{n}$ be nonzero positive real numbers.
Moreover, suppose that these numbers satisfy the following inequality 
\begin{equation*}
0\leq 1-\frac{2\beta +1}{\beta ^{2}-\beta -1}\sum\limits_{j=1}^{n}\frac{1}{%
\lambda _{j}}<1.
\end{equation*}%
Then the function $\mathbb{F}_{\alpha _{j},\beta _{j},\lambda _{j}}$ defined
by 
\begin{equation}
\mathbb{F}_{\alpha _{j},\beta _{j},\lambda
_{j}}(z)=\int\limits_{0}^{z}\prod_{j=1}^{n}\left( \frac{\mathbb{E}_{\alpha
_{j},\beta _{j}}(t)}{t}\right) ^{1/\lambda _{j}}dt,  \label{v2}
\end{equation}%
is in $\mathcal{C(}\delta \mathcal{)}$, where 
\begin{equation*}
\delta =1-\frac{2\beta +1}{\beta ^{2}-\beta -1}\sum\limits_{j=1}^{n}\frac{1}{%
\lambda _{j}}.
\end{equation*}
\end{theorem}

\begin{proof}
We observe that $\mathbb{E}_{\alpha _{j},\beta _{j}}\in \mathcal{A}$, i.e. $%
\mathbb{E}_{\alpha _{j},\beta _{j}}(0)=\mathbb{E}_{\alpha _{j},\beta
_{j}}^{\prime }(0)-1=0$, for all $j\in \{1,2,\dots ,n\}.$ On the other hand,
it is easy to see that 
\begin{equation*}
\mathbb{F}_{\alpha _{j},\beta _{j},\lambda _{j}}^{\prime
}(z)=\prod_{j=1}^{n}\left( \frac{\mathbb{E}_{\alpha _{j},\beta _{j}}(z)}{z}%
\right) ^{1/\lambda _{j}}
\end{equation*}%
and 
\begin{equation*}
\frac{z\mathbb{F}_{\alpha _{j},\beta _{j},\lambda _{j}}^{\prime \prime }(z)}{%
\mathbb{F}_{\alpha _{j},\beta _{j},\lambda _{j}}^{\prime }(z)}=\sum_{j=1}^{n}%
\frac{1}{\lambda _{j}}\left( \frac{z\mathbb{E}_{\alpha _{j},\beta
_{j}}^{\prime }(z)}{\mathbb{E}_{\alpha _{j},\beta _{j}}(z)}-1\right) ,
\end{equation*}

or, equivalently,

\begin{equation}
1+\frac{z\mathbb{F}_{\alpha _{j},\beta _{j},\lambda _{j}}^{\prime \prime }(z)%
}{\mathbb{F}_{\alpha _{j},\beta _{j},\lambda _{j}}^{\prime }(z)}%
=\sum_{j=1}^{n}\frac{1}{\lambda _{j}}\left( \frac{z\mathbb{E}_{\alpha
_{j},\beta _{j}}^{\prime }(z)}{\mathbb{E}_{\alpha _{j},\beta _{j}}(z)}%
\right) +1-\sum_{j=1}^{n}\frac{1}{\lambda _{j}}.  \label{b6}
\end{equation}

Taking the real part of both terms of (\ref{b6}), we have%
\begin{equation}
\mbox{Re}\left\{ 1+\frac{z\mathbb{F}_{\alpha _{j},\beta _{j},\lambda
_{j}}^{\prime \prime }(z)}{\mathbb{F}_{\alpha _{j},\beta _{j},\lambda
_{j}}^{\prime }(z)}\right\} =\sum_{j=1}^{n}\frac{1}{\lambda _{j}}\mbox{Re}%
\left( \frac{z\mathbb{E}_{\alpha _{j},\beta _{j}}^{\prime }(z)}{\mathbb{E}%
_{\alpha _{j},\beta _{j}}(z)}\right) +\left( 1-\sum_{j=1}^{n}\frac{1}{%
\lambda _{j}}\right) .  \label{b7}
\end{equation}

Now, by using the inequality \eqref{r4} for each $\beta _{j},$ where $j\in
\{1,2,\dots ,n\},$ we obtain 
\begin{eqnarray*}
\mbox{Re}\left\{ 1+\frac{z\mathbb{F}_{\alpha _{j},\beta _{j},\lambda
_{j}}^{\prime \prime }(z)}{\mathbb{F}_{\alpha _{j},\beta _{j},\lambda
_{j}}^{\prime }(z)}\right\} &=&\sum_{j=1}^{n}\frac{1}{\lambda _{j}}\mbox{Re}%
\left( \frac{z\mathbb{E}_{\alpha _{j},\beta _{j}}^{\prime }(z)}{\mathbb{E}%
_{\alpha _{j},\beta _{j}}(z)}\right) +\left( 1-\sum_{j=1}^{n}\frac{1}{%
\lambda _{j}}\right) \\
&>&\sum_{j=1}^{n}\frac{1}{\lambda _{j}}\left( 1-\frac{2\beta _{j}+1}{\beta
_{j}^{2}-\beta _{j}-1}\right) +\left( 1-\sum_{j=1}^{n}\frac{1}{\lambda _{j}}%
\right) \\
&=&1-\frac{2\beta +1}{\beta ^{2}-\beta -1}\sum_{j=1}^{n}\frac{1}{\lambda _{j}%
}
\end{eqnarray*}%
for all $z\in \mathbb{D}$ and $\beta _{1},\beta _{2},\dots ,\beta _{n}\geq 
\frac{1}{2}(1+\sqrt{5}).$ Here we used that the function $\varphi :(\frac{1}{%
2}(1+\sqrt{5}),\infty )\rightarrow \mathbb{R},$ defined by 
\begin{equation*}
\varphi (x)=\frac{2x+1}{x^{2}-x-1},
\end{equation*}%
is decreasing. Therefore, for all $j\in \{1,2,\dots ,n\}$ we have 
\begin{equation}
\frac{2\beta _{j}+1}{\beta _{j}^{2}-\beta _{j}-1}\leq \frac{2\beta +1}{\beta
^{2}-\beta -1}.  \label{ww}
\end{equation}

Because $0\leq 1-\frac{2\beta +1}{\beta ^{2}-\beta -1}\sum\limits_{j=1}^{n}%
\frac{1}{\lambda _{j}}<1,$we get $\mathbb{F}_{\alpha _{j},\beta _{j},\lambda
_{j}}(z)\in \mathcal{C(\delta )}$, where $\mathcal{\delta =}1-\frac{2\beta +1%
}{\beta ^{2}-\beta -1}\sum\limits_{j=1}^{n}\frac{1}{\lambda _{j}}.$ This
completes the proof.
\end{proof}

Let $n=1,$ $\alpha _{1}=\alpha ,$ $\beta _{1}=\beta $ and $\lambda
_{1}=\lambda $ in Theorem \ref{th1}, we have the following result.\bigskip

\begin{corollary}
Let $\alpha \geq 1$, $\beta \geq \frac{1}{2}(1+\sqrt{5})$ and $\lambda >0$.
Moreover, suppose that these numbers satisfy the following inequality 
\begin{equation*}
0\leq 1-\frac{2\beta +1}{\lambda (\beta ^{2}-\beta -1)}<1.
\end{equation*}%
Then the function $\mathbb{F}_{\alpha ,\beta ,\lambda }$ defined by 
\begin{equation}
\mathbb{F}_{\alpha ,\beta ,\lambda }(z)=\int\limits_{0}^{z}\left( \frac{%
\mathbb{E}_{\alpha ,\beta }(t)}{t}\right) ^{1/\lambda }dt,
\end{equation}%
is in $\mathcal{C(}\delta \mathcal{)}$, where 
\begin{equation*}
\delta =1-\frac{2\beta +1}{\lambda (\beta ^{2}-\beta -1)}.
\end{equation*}
\end{corollary}

\begin{example}
\bigskip $(i)~$If $0\leq 1-\frac{5}{\lambda }<1,$ then $\int\limits_{0}^{z}%
\left( \frac{\sinh \sqrt{t}}{\sqrt{t}}\right) ^{1/\lambda }dt\in \mathcal{C(}%
\delta \mathcal{)};\delta =1-\frac{5}{\lambda };\lambda \geq 5.$
\end{example}

$(ii)~$\textit{If \ }$0\leq 1-\frac{7}{5\lambda }<1,$\textit{\ then} $%
\int\limits_{0}^{z}\left( \frac{2[\cosh \sqrt{t}-1]}{t}\right) ^{1/\lambda
}dt\in \mathcal{C(}\delta \mathcal{)};\delta =1-\frac{7}{5\lambda };\lambda
\geq 7/5.$

$(iii)$\textit{\ If }$0\leq 1-\frac{9}{11\lambda }<1,$\textit{\ then} $%
\int\limits_{0}^{z}\left( \frac{6[\sinh \sqrt{t}-\sqrt{t}]}{t^{3/2}}\right)
^{1/\lambda }dt\in \mathcal{C(}\delta \mathcal{)};\delta =1-\frac{9}{%
11\lambda };\lambda \geq 9/11.$

\end{document}